\documentclass{amsart} 

\usepackage{graphicx}
\usepackage{youngtab}
\usepackage[dvipsnames]{xcolor}
\usepackage{hyperref}
\usepackage{blkarray}
\usepackage{amssymb}

\newtheorem{theorem}{Theorem}[section]
\newtheorem{lemma}[theorem]{Lemma}

\theoremstyle{definition}

\newtheorem{examplex}[theorem]{Example}
\newenvironment{example}
  {\pushQED{\qed}\examplex}
  {\popQED\endexamplex}

\theoremstyle{remark}

\numberwithin{equation}{section}

\newcommand{\R}{\mathcal{R}}
\newcommand{\N}{\mathcal{N}}

\definecolor{mygreen}{RGB}{0, 102, 0}

\newcommand{\newword}[1]{\textbf{\textit{#1}}}

\newcommand{\F}{\mathcal{F}}
\newcommand{\Q}{\mathbb{Q}}
\newcommand{\FF}{\mathbb{F}}

\newcommand{\QQ}{\mathbb{Q}}

\newcommand{\ww}{\mathbf{w}}
\newcommand{\rr}{\mathbf{r}}
\newcommand{\pp}{\mathbf{p}}
\newcommand{\End}{\textup{End}}
\newcommand{\Proj}{\textup{Proj}}
\newcommand{\smsn}{S_m \wr S_n}

\begin{document}

\title{Algebraic Voting Theory \& Representations of $S_m \wr S_n$}

\author[Barcelo, Bernstein, Bockting-Conrad, McNicholas, Nyman, Viel]{H{\'e}l{\`e}ne Barcelo, Megan Bernstein, Sarah Bockting-Conrad, Erin McNicholas, Kathryn Nyman, Shira Viel}\thanks{This research was partially supported by NSF ADVANCE Grant \#1500481, DMS-1440140, DMS-1344199, and DMS-1500949.}

\address{
H{\'e}l{\`e}ne Barcelo\\
Mathematical Sciences Research Institute}
\email{hbarcelo@msri.org}
 
\address{
Megan Bernstein\\
Georgia Institute of Technology}
\email{bernstein@math.gatech.edu}

\address{
Sarah Bockting-Conrad\\
DePaul University}
\email{sarah.bockting@depaul.edu}

\address{
Erin McNicholas\\
Willamette University}
\email{emcnicho@willamette.edu}

\address{
Kathryn Nyman\\
Willamette University}
\email{knyman@willamette.edu}

\address{
Shira Viel\\
Duke University}
\email{shira.viel@duke.edu}

\maketitle

\begin{abstract} 
We consider the problem of selecting an $n$-member committee made up of one of $m$ candidates from each of $n$ distinct departments. Using an algebraic approach, we analyze positional voting procedures, including the Borda count, as $\mathbb{Q}S_m \wr S_n$-module homomorphisms.
In particular, we decompose the spaces of voter  preferences and election results into simple $\mathbb{Q}S_m \wr S_n$-submodules and apply Schur's Lemma to determine the structure of the information lost in the voting process.
We conclude with a voting paradox result, showing that for sufficiently different weighting vectors, applying the associated positional voting procedures to the same set of votes can yield arbitrarily different election outcomes.
\end{abstract}

\section{Introduction}
In this paper, we examine the process of electing an $n$-member committee comprised of one representative from each of $n$ departments, 
with each departmental representative chosen from a field of $m$ candidates.
Voters give complete or partial rankings of the set of possible committees from most to least desired. 
Using a positional voting procedure, points are awarded to committees based on their position in each voter's ranking, where a choice of weighting vector determines the amount of points awarded for each position. The committee with the most points is elected. 

The motivation for voters ranking all possible committees stems from Ratliff's work that shows voter preferences in the committee selection setting are often more complex and nuanced than their rankings of the candidates \cite{Ratliff}.  For example, voters can be influenced by a desire for diverse representation across the committee or consideration for the relationships among committee members.  
Examining data from a 2003 university election, Ratliff found strong evidence that voter committee preferences would not have been captured by having voters simply select their favorite candidate from each division.  In particular, more than half the voters in the election had first and last choice committee preferences that were not disjoint.

In what follows, we study committee voting procedures by examining the underlying algebraic structures of two spaces: the space of voter preferences and the space of possible election outcomes. This algebraic approach to voting theory was first introduced by Daugherty, Eustis, Minton, and Orrison \cite{DEMO}. 
Using representations of the symmetric group, Daugherty, et.~al., recover and extend several results of Saari, a pioneer in the field of mathematical voting theory who studied the geometry and subspaces of voting information (see, e.g., \cite{Saari}). The purpose of such analysis is to uncover what aspects of voter preferences 
contribute to the outcome of an election when a positional voting method is used.

Following the lead of Daugherty, et.~al., Lee introduced the idea of using representations of wreath products of symmetric groups to study committee elections \cite{Lee}. In this setting, the spaces of voter preferences and possible election results are viewed as wreath product modules.  
Our contribution is to find the module decomposition of these spaces, and to otherwise extend the results in \cite{DEMO} to committee selection voting.

In particular, we explore a paradox that arises from the fact that there is no voting system that perfectly translates the preferences of voters into a single election outcome and that results often say as much about the method of voting as they do about voter preference.
To show  
how an election procedure can affect the election outcome when voting for individual candidates, Saari proved that given $j$ sufficiently different weighted positional voting systems and any $j$ orderings of $n$ candidates, $A_1, \ldots, A_j$, there exist examples of voter preferences such that using voting system $i$ results in the ranking of candidates $A_i$  \cite[Theorem~1]{Saari}. Essentially, this implies that if the weighting vectors are sufficiently different, the corresponding election results might not resemble each other at all despite using the same set of voter preferences.

In \cite[Theorem~1]{DEMO}, Daugherty, et.~al. prove a stronger result using their algebraic framework.  They show that given $j$ sufficiently different 
weighting vectors and any $j$ corresponding vectors of point totals for each candidate, there exist infinitely many voter preferences such that using weighting vector $i$ results in the corresponding point total vector, for all $1\leq i\leq j$. Notice here that not only does the ranking of the candidates rely on the voting procedure, but so too does any (relative) difference in point totals.

We obtain an analog of \cite[Theorem~1]{DEMO} for committee voting, showing again that if a set of weighting vectors are sufficiently different,   
 the associated positional voting procedures can yield radically different outcomes. 
 Because the algebraic structure of the profile and results spaces are more complex in our case, the conditions for what constitute ``sufficiently different'' weighting vectors are stronger.

\section{Algebraic Voting Theory Background}
\label{sec: alg voting background}

In this section we review the basics of voting theory, first from a geometric perspective and then through the lens of representation theory, framing definitions in the context of committee voting where appropriate. For a more detailed treatment of this content in the case of electing a single candidate, we refer the reader to  \cite[Sections~2-3]{DEMO}.

\subsection{Voting theory from a geometric point of view}  We begin by introducing standard definitions of voting theory.
These definitions allow us to view certain types of voting procedures as linear transformations from the space of voter preferences, to the space of possible election outcomes.

 A \newword{positional voting procedure} is a voting method 
 in which each voter ranks the slate of committees, and points are then assigned to committees based on their positions within these rankings.  For example, in an election with $d$ total committees, the \newword{Borda Count} assigns $d-1$ points for a first place position in a voter's ranking, $d-2$ points for a second place position, and so on, down to $0$ points for a last place position.  The \newword{score} for a given committee is the sum, over all voter rankings, of the number of points awarded to that committee given its position. 
The committee with the highest score wins the election. 
The election outcome is encoded by a \newword{results vector} $\mathbf{r}$, a rational column vector indexed by the set of committees whose entries encode each committee's score. 
  
In general, we can assign a rational \newword{weighting vector} $\mathbf{w}$ to each positional voting procedure based on the number of points awarded for each position. The weighting vector for the Borda Count, viewed as a column vector in $\Q^d$, is  $\mathbf{w}=[d-1,d-2,\ldots,1,0]^t$. 
 
Consider the process of using a positional voting procedure to select an $n$-member committee which consists of one of $m$ candidates from each of $n$ distinct departments. In particular, let $[n]$ denote the set $\{1,2, \ldots, n\}$, and enumerate the set of all candidates from all departments as
\begin{equation}
\label{eq: candidate set}
X=\{i_k: i \in [n] \textup{ and }k \in [m]\},
\end{equation} 
where $i_k$ denotes candidate $k$ from department $i$.
The election slate consists of $m^n$ committees, enumerated by subsets of $X$ of the form $\{1_{k_1},2_{k_2},\ldots, n_{k_n}\}$ where $k_j \in [m]$ for each $j\in [n]$. Because these subsets must always consist of a single representative from each department, we generally refer to committees as ordered $n$-tuples $(1_{k_1},2_{k_2},\ldots, n_{k_n})$.

There are $(m^n)!$ possible full rankings of committees, and to determine the winner of the election, we need to know the number of voters who selected each ranking.
This information can be encoded by a \newword{profile} $\mathbf{p}$; a column vector indexed by the set of full rankings with whole number entries recording the number of votes each ranking receives. 
Together, a profile and a weighting vector determine the outcome of an election under a positional voting procedure. The profile, by encoding the number of votes for each ranking, records the number of first place positions, second place positions, and so on allotted to a given committee. The weighting vector determines the number of points assigned to each position. 
By generalizing the definition of a profile to include vectors with {\em rational} entries, the set of all profiles forms a vector space over $\QQ$.
This vector space, called the \newword{profile space} $P$, contains the preferences of all possible electorates and 
has a natural basis given by the set of $(m^n)!$ full rankings of committees.  Similarly, the set of all results vectors span a vector space over $\Q$, called the \newword{results space} $R$, with basis given by the set of $m^n$ committees.

In this paper, we focus exclusively on the case where voters provide 
full rankings of the possible committees, but emphasize that this model can be tailored to positional voting procedures which use partial rankings (such as voting for a favorite committee or choosing one's favorite three committees) by choosing an appropriate weighting vector. For example, in the case where voters are asked for their favorite committee, the rest of the ranking can be filled in randomly and the appropriate result
 can be obtained by using the weighting vector $\mathbf{w}=[1,0,0,\ldots,0]^t$. 
 When voters are asked to rank their top three committees, the weighting vector  $\mathbf{w}=[3,2,1,0,0,\ldots,0]^t$ yields the desired result.

 \begin{example}\label{ex:2member}
Consider an election for a 2-member committee consisting of one of two candidates from each of departments 1 and 2. For ease, we will avoid subscripts by  
referring to the candidates from department 1 as   
$A$ and $B$, and the candidates from department 2 as  $a$ and $b$. 
We also avoid set or vector notation by referring to a committee simply by its members, listing the candidate from department 1 first. 
There are 4 possible committees and 4! full rankings of the committees, which we order lexicographically (in particular, as $Aa, Ab, Ba, Bb$) to form ordered bases for $R$ and $P$, respectively. The Borda count weighting vector is $\mathbf{w}=[3,2,1,0]^t$. 

Suppose we have 9 voters, four of whom rank the committees $Aa> Ab> Ba>Bb$, three of whom rank the committees $Aa>Ba>Ab>Bb$, and two of whom rank the committees $Ab>Aa>Bb>Ba$. (All other rankings receive zero votes.)  The profile vector $\mathbf{p}$ encoding these preferences is given below, with labels corresponding to the ordered basis of committee rankings written on the right for convenience. 
Since the committee $Ab$ is allotted 2 first place positions, 4 second place positions, 3 third place positions, and 0 last place positions,
$Ab$'s Borda score is $2\times 3+ 4\times 2 + 3 \times 1 + 0 \times 0=17$.  The Borda scores for $Aa$, $Ba$, and $Bb$ can be obtained similarly. 
$$\mathbf{p} = \begin{bmatrix}
4 \\
0 \\
3 \\
0 \\
0 \\
0\\
0\\
2\\
\vdots
\end{bmatrix}
\begin{matrix}
Aa>Ab>Ba>Bb \\
Aa>Ab>Bb>Ba \\
Aa>Ba>Ab>Bb \\
Aa>Ba>Bb>Ab \\
Aa>Bb>Ab>Ba\\
Aa>Bb>Ba>Ab\\
Ab>Aa>Ba>Bb\\
Ab>Aa>Bb>Ba\\
\vdots \\
\end{matrix} 
\hspace{.05in}
\xrightarrow[{\hspace{.03in}\mathbf{w}=[3,2,1,0]^t}\hspace{.03in}]{\text{\hspace{.05in}Borda \hspace{.05in}}} 
\hspace{.05in}
\mathbf{r}=\begin{bmatrix} 25\\17\\10\\2\end{bmatrix}  
\begin{matrix}
Aa \\
Ab \\
Ba\\
Bb
\end{matrix}.$$

This positional voting procedure with weighting vector $\mathbf{w}$ can be represented by a linear transformation
$T_{\mathbf{w}}: P \to R$ 
mapping a profile $\mathbf{p}$ encoding the number of votes received by each full ranking to the results vector $\mathbf{r}$ encoding the number of points assigned to each committee as a result.
Viewing $T_{\mathbf{w}}$ as an $m^n \times (m^n)!$ matrix with rational entries, 
the columns of $T_{\mathbf{w}}$ are permutations of the weighting vector $\mathbf{w}$, 
ordered to correspond to the ordered bases of full rankings and committees for $P$ and $R$, respectively.
Specifically, the $(i,j)$ entry of $T_{\mathbf{w}}$ is the number of points assigned by $\mathbf{w}$ to committee $i$ based on its position in the $j^{th}$ lexicographical ranking. In our example, we have

\[T_{[3,2,1,0]}(\mathbf{p}) =
\begin{bmatrix}
  3 & 3 & 3 & 3 & 3 & 3 & 2 & 2 \cdots\\
  2 & 2 & 1 & 0 & 1 & 0 & 3 & 3 \cdots\\
  1 & 0 & 2 & 2 & 0 & 1 & 1 & 0 \cdots \\
  0 & 1 & 0 & 1 & 2 & 2 & 0 & 1 \cdots\\
  \end{bmatrix}
\begin{bmatrix}
4 \\
0 \\
3 \\
0 \\
0 \\
0\\
0\\
2\\
\vdots
\end{bmatrix}
=
\begin{bmatrix} 25\\17\\10\\2\end{bmatrix}  
\begin{matrix}
Aa \\
Ab \\
Ba\\
Bb
\end{matrix}.
 \]
 \end{example}

Here we note that the weighting vector $\mathbf{w}$ can be thought of as an element of the results space $R$, as it is a rational vector indexed by the set of committees. We can decompose $\mathbf{w}$ into the sum of a scalar multiple of the all-ones vector $\mathbf{1}$ and an orthogonal vector $\hat{\mathbf{w}}$ whose entries sum to zero. That is,
 $\mathbf{w}=\alpha \mathbf{1} + \hat{\mathbf{w}}$.
We call the vector $\hat{\mathbf{w}}$ a \newword{sum-zero} vector. 
For example, 
$$\mathbf{w}= [3,2,1,0]^t = \frac{3}{2}\mathbf{1} +  \left[\frac{3}{2},\frac{1}{2},-\frac{1}{2},-\frac{3}{2}\right]^t.$$ 
In general, the component $\alpha \mathbf{1}$ 
contributes equally to each committee, while the sum-zero component records the relative differences in the number of votes received by each committee. Hence the sum-zero component contains all of the information determining the election results.

\subsection{Voting theory through an algebraic lens}

A common requirement of voting procedures is that they be \newword{neutral}: that is, the outcome of the election must not depend on the presentation of the candidates on the ballot.  If the names of the candidates are permuted, the results of the election will be permuted in the same way. For example, suppose that the winning ranking in an election with three candidates is $A>B>C$, and we switch the labels of candidates $A$ and $C$. If the voting procedure is neutral, then the winning ranking would be $C>B>A$. 
Thus, neutrality defines a required symmetry within the voting system and suggests the utility of applying an algebraic lens. 

When selecting a winner from a slate of $n$ candidates,
permuting the names of the candidates corresponds to acting on the set of candidates by an element $\sigma$ of the symmetric group  $S_n$. Consequently, a neutral voting procedure requires that applying $\sigma$ to the list of candidates and then performing the voting procedure yields the same outcome as first performing the voting procedure and then 
applying $\sigma$ to the results vector.

In the case of committee voting,
however, a positional voting procedure is neutral not if we can relabel all $m^n$ committees arbitrarily, but if we can relabel the $n$ departments and $m$ candidates within each department independently without affecting the election outcome.  This symmetry requirement is encapsulated in the 
natural action of the wreath product of the symmetric group $S_m$ with $S_n$ on the set of candidates $X$ enumerated in \eqref{eq: candidate set}. 
We use the notation $S_m \wr S_n$ to denote this wreath product:
\begin{equation}
    S_m \wr S_n := \{(\sigma;\pi):\sigma \in S_m^n \textup{ and }\pi \in S_n \}.
\end{equation}
The action of an element $(\sigma;\pi) \in S_m \wr S_n$ on the set of candidates $X$ first permutes the departments and then permutes the candidates within each department. 
We view $\pi$ as acting on $\sigma=(\sigma_1, \sigma_2, \ldots, \sigma_n) \in S_m^n$ by coordinate permutation. 
The action of $S_m \wr S_n$ on 
$i_k \in X$ (candidate $k$ from department $i$) is then defined as follows:
\begin{equation}\label{wraction}
   (\sigma; \pi)i_k = \pi(i)_{\sigma_{\pi(i)}(k)}.
\end{equation}

Since committees are subsets of $X$, the action of $S_m \wr S_n$ on $X$ extends naturally to an action on the set of committees. 
We note that the committee $(1_{k_1}, 2_{k_2}, \ldots, n_{k_n})$ is sent to committee $(1_{\sigma_1(k_{\pi^{-1}(1)})}, 2_{\sigma_2(k_{\pi^{-1}(2)})},\ldots, n_{\sigma_n(k_{\pi^{-1}(n)})})$
We illustrate this action with an example.

\begin{example}
Recall the setup from Example \ref{ex:2member}, which examined the election of a 2-member committee consisting of one of two candidates from each of two departments. We again denote the candidates from department 1 by $A,B$ and the candidates from department 2 by $a,b$. 
(In subscript notation, $A=1_1, B=1_2, a=2_1,$ and $b=2_2$.)
Consider the permutation $(\sigma;\pi)=((12),e;(12)) \in S_2 \wr S_2$ applied to the committee $Ab$.
First $\pi=(12)$ indicates that the departments switch labels. 
We can think of this as relabeling capital letters with lowercase letters and vice versa, so that $\pi(Ab) = aB = Ba$. (Recall committees are {\em sets} of candidates, so the order of candidates does not matter. By convention, we list the candidate from department 1 first.) 
Next, we apply $\sigma_1=(12)$ to the candidate from department 1, $B$, and $\sigma_2=e$ to the candidate from department 2, $a$, resulting in the action  $(\sigma;\pi)(Ab)=\sigma(Ba)=\sigma_1(B)\sigma_2(a)=Aa$.

To see why the symmetric group $S_4$ is not the appropriate group to capture the symmetry requirement imposed by neutrality, consider the element $\tau=(12) \in S_4$ acting on the full ranking 
of committees $Aa>Ab>Ba>Bb$. While $\tau(Aa>Ab>Ba>Bb)$ gives the ranking $Ab>Aa>Ba>Bb$, notice that no renaming of the individual candidates would result in this new ranking.  The first place committee is now $Ab$, rather than $Aa$, suggesting that either $a$ was renamed $b$, or $a$ was renamed $A$ and $A$ was renamed $b$. However, the committee in third place remains $Ba$, and so no such relabeling has occurred. 
\end{example}

This action of the wreath product on committees extends to an action on the results space $R$ for which they form a basis and induces an action on the set of full rankings of committees, which extends to an action on the profile space $P$ for which they form a basis. 
We may therefore view $R$ and $P$ as $\QQ S_m \wr S_n$-modules. The neutrality condition  means that the voting procedure map $T_{\mathbf{w}}$ is a $\QQ S_m \wr S_n$-module homomorphism (see \cite{DEMO} and \cite{Lee}), thereby placing the tools of representation theory at our disposal.  In particular, if we can decompose the profile and results spaces into  simple submodules, we can apply Schur's Lemma to the voting procedure map $T_{\mathbf{w}}$.
\vspace{.1in}

\noindent \textbf{Schur's Lemma.}
Every nonzero module homomorphism between simple modules is an isomorphism.
\vspace{.1in}

By comparing the simple submodule decompositions of the profile and results spaces,   
we can see the structure of the information in the profile space that is pivotal in determining the outcome of the election, as well as the structure of the information that has no impact on the election result--namely, the kernel of  $T_{\mathbf{w}}$.

\section{Profile and Results Space Decompositions}\label{sec: space decomposition}

In order to apply Schur's Lemma, we must decompose the profile and results spaces into simple (i.e., irreducible) submodules. In this section, we use the correspondence between modules and group representations to achieve this goal.  Since the representations of the wreath product $S_m \wr S_n$ corresponding to the module structure of our profile and results spaces are over the ground field $\Q$, which has characteristic zero, they are uniquely determined (up to isomorphism) by their character. By 
computing the characters of the representations on the profile and results spaces, and matching them with characters of known and well-studied representations of the wreath product, we are able to determine the irreducible decompositions of these spaces.

Finding a nice formula for a character of a representation is typically a difficult problem. Often the best that can be done is to obtain a recursive formula like the Murnaghan-Nakayama rule for the characters of the irreducible representations of the symmetric group \cite{Stanley}. 
However, in our situation, we can compute the character relatively easily because we have nice bases for the profile and results spaces (namely the set of full rankings of committees and the set of committees respectively), and we understand how $S_m \wr S_n$ acts on these bases.

Let $X$ denote a finite set.  If a group $G$ acts on $X$ via automorphisms of $X$, then there is an associated \newword{permutation representation} of $G$, defined as follows.  Let $V$ denote a vector space over the field $\FF$ with basis $\{b_x|x\in X\}$. We define the action of $G$ on $V$ to be $$g\cdot\sum_{x\in X} a_xb_x=\sum_{x\in X} a_xb_{gx}$$ for all $g\in G$ and $a_x\in\FF$.  Thus, for all $g\in G$, the image of a permutation representation $\rho$, is a permutation matrix, and $\chi_{\rho}(g)$, which is the sum of the diagonal entries of $\rho(g)$, is equal to the number of fixed points of the action of $g$ on $X$.  We will use this fact shortly when discussing the characters of wreath product representations.

If $X=G$, the group operation induces a permutation representation called the  \newword{regular representation} of $G$  \cite[p.5]{FH}. 
In the case of symmetric groups, the natural action of $S_m$ on the set $X=\{1,2,\mathellipsis, m\}$ induces a permutation representation called the \newword{natural representation} of $S_m$.  From the discussion above, for all $\sigma\in S_m$, the character $\chi_{\rho}(\sigma)$ of the natural representation counts the number of fixed points of $\sigma.$

 \subsection{Representation corresponding to the profile space}
\label{sec:profile_space}
In this section, we use the fact that the character of a permutation representation assigns to each group element $g$ of $G$ the number of fixed points in $X$ under the action of $g,$  to show that the profile space $P$ decomposes into the direct sum 
of $\frac{m^n!}{(m!)^n n!}$ copies of the regular 
$\QQ S_m \wr S_n$-module.
What follows is a generalization of the 
$S_2 \wr S_2$ case discussed in \cite[Section~4.1]{Lee}.

\begin{theorem}
\label{thm: profile space decomposition1}
The profile space $P$ for the full-ranking voting procedure decomposes into the direct sum of $\frac{m^n!}{(m!)^n n!}$ copies of the regular $\QQ S_m \wr S_n$-module. That is,
\begin{equation}
   P \cong \bigoplus_{\frac{m^n!}{(m!)^n n!}} \R_{S_m \wr S_n},
\end{equation}
where $\R_{S_m\wr S_n}$ is the regular $\QQ S_m \wr S_n$-module.
\end{theorem}

\begin{proof}
The set of full rankings which form the basis of $P$ are partitioned into orbits under the action of $S_m \wr S_n$. 
For each full ranking $\F$, let $P_{\F}$ denote the subspace of $P$ spanned by the orbit $O_{\F}$ of $\F$ under action by $S_m \wr S_n$. 
That is, 
\begin{equation*}
P_{\F} = \textup{span}(O_{\F}) =
\textup{span}\{(\sigma;\pi)(\F):(\sigma;\pi) \in S_m \wr S_n\}.
\end{equation*}
Because the $S_m \wr S_n$-orbits partition the basis of $P$, it is clear that as  a vector space, $P$ is the direct sum, over all $S_m \wr S_n$-orbits, of the subspaces they span:
\begin{equation}
\label{eq: profile space eq1}
    P \cong \bigoplus_{S_m \wr S_n\textup{-orbits } O_{\F}}P_{\F}.
\end{equation}
Further, each subspace $P_{\F}$ is a $\QQ S_m \wr S_n$ permutation module since the action of $S_m \wr S_n$ permutes its basis elements. Thus \eqref{eq: profile space eq1} holds as a $\QQ S_m \wr S_n$-module decomposition as well.

Next, note that only the identity element of $S_m \wr S_n$ fixes every candidate (in $X$), and therefore every committee. It follows that no full ranking is fixed by any non-identity group elements. Since the character of the representation of $S_m \wr S_n$ corresponding to a submodule $P_{\F}$ evaluated at any group element $(\sigma;\pi)$ counts the number of full rankings in $P_{\F}$ fixed by the action of $(\sigma;\pi)$, we have
\begin{equation}
    \chi_{P_{\F}}(\sigma;\pi) = 
    \begin{cases} 
    |S_m \wr S_n| & \textup{if } (\sigma;\pi) = (e;e) \\
    0 & \textup{ otherwise}
    \end{cases}
\end{equation}

These are precisely the same results as those obtained by evaluating the character of the regular representation of 
$S_m \wr S_n.$  Since $\QQ$ is a field of characteristic zero, this implies that each submodule $P_{\F}$ is isomorphic to 
$\R_{S_m\wr S_n}$. 

Finally, we count the number of copies of the regular module that appear in the decomposition of $P$ on the right-hand side of \eqref{eq: profile space eq1}. 
The dimension of $P$ is $m^n!$ and each subspace $P_{\F}$ has dimension $|S_m \wr S_n|=m!^n n!$, so there are exactly $\frac{m^n!}{m!^n n!}$ $S_m \wr S_n$-orbits $O_{\F}$ and therefore $\frac{m^n!}{m!^n n!}$ copies of $\R_{S_m\wr S_n}$. 
\end{proof}

 \subsection{Representation corresponding to the results space}\label{sec:result_rep}
 
Next, we calculate the character of the wreath product representation corresponding to the results space by counting the number of committees fixed by elements of the wreath product. 

Consider an element $(\sigma;\pi)\in\smsn.$ 
For a committee $\textbf{c}=(1_{k_1},2_{k_2},\ldots, n_{k_n})$, the wreath product element $(\sigma;\pi)$ acts on $\textbf{c}$ as
 \[(\sigma; \pi)\textbf{c} = (1_{\sigma_1(k_{\pi^{-1}(1)})}, \ldots, n_{\sigma_n(k_{\pi^{-1}(n)})}).\]
  Thus, a candidate $i_{k_i}$ from the original committee will  remain in the committee if there exists a $j$ such that 
 \[(\sigma;\pi)j_{k_j} = \pi(j)_{\sigma_{\pi(j)}(k_j)} = i_{k_i}.\]

In order for a committee $\textbf{c}$ to be fixed under the action of an element $(\sigma;\pi) \in S_m \wr S_n$, it must be that, for each $i$, the orbit of candidate $i_k$ under $(\sigma;\pi)$ has the same size as the length of the cycle in $\pi$ containing $i$. To see this, notice that the cycle of $\pi$ containing $i$ gives a directed ordering of how candidates are mapped on a departmental level. For example, if $\pi = (143)(25)$, this indicates that candidates in department 1 are sent to candidates in department 4 which, in turn, are mapped to candidates in department 3, which are finally sent back to department 1. Next, consider following candidate $k$ from department 1.  Since the $\sigma$'s act on candidates after the departmental shift, we find candidate $k$ from department 1 is sent to candidate $\sigma_4(k)$ in department 4, which is sent to candidate $\sigma_3(\sigma_4(k))$ in department 3 which is sent back to department 1 as candidate $\sigma_1(\sigma_3(\sigma_4(k)))$. Therefore, in order for candidate $1_k$ to appear in a committee fixed by $(\sigma;\pi)$, the orbit of $1_k$ must return to department 1 as candidate $k$. In other words, candidate $i_k$ appears in a fixed committee if $\sigma_i \sigma_{\pi^{-1}(i)} \cdots \sigma_{\pi^{-l+1}(i)}(k) =k$, where $l$ is the length of the cycle in $\pi$ which contains $i$.

For ease of notation, let $c(\pi)$ denote the number of disjoint cycles of $\pi$. For some ordering of these cycles, let $l(\nu)$ denote the length of the $\nu$th cylce. If department $i$ is an element of the $\nu$th cycle, the cycle can be expressed as $(i, \pi(i), \cdots, \pi^{l(\nu)-1}(i))$. For each $\nu$, the corresponding cycle product of $(\sigma; \pi)$, used in the determination of fixed points, is denoted 
\begin{equation}\label{eq:gnu}
g_\nu(\sigma; \pi) = \sigma_{i} \sigma_{\pi^{-1}(i)} \cdots \sigma_{\pi^{-l(\nu) +1}(i)}.
\end{equation}
Note that $g_{\nu}$ is an element of $S_m.$  In our definition of $g_{\nu}$, $i$ was an arbitrary department appearing in the $\nu$th cycle.  If we are starting with a particular department and want to refer to the cycle product corresponding to the subcycle of $\pi$ in which $i$ appears, we denote that $g_{\nu_i}.$

Thus, a committee $\textbf{c}$ is fixed by an element $(\sigma;\pi)$ precisely when $g_{\nu_i}(k_i)=k_i$ for all $i$ in $[n].$ 
This means for each $i$, $k_i$ can be any fixed point of $g_{\nu_i}$. Selecting a particular $k_i$ determines all the candidates for the remaining departments in the cycle $\nu_i$. This is illustrated in the following example.

\begin{example}\label{ex:gnu}
Here we compute the fixed committees for several elements of wreath products. First consider $(\sigma; \pi)=((123),(12)(3);(12)) \in S_3 \wr S_2$.  
The action of $((123),(12)(3);(12))$ sends candidate $1_1$ to $2_2$, $1_2$ to $2_1$, $1_3$ to $2_3$, $2_1$ to $1_2$, $2_2$ to $1_3$, and $2_3$ to $1_1$. 
Following the orbit of candidate $1_1$, we find $1_1 \rightarrow 2_2 \rightarrow 1_3 \rightarrow 2_3 \rightarrow 1_1$. Its orbit has size 4, but 1 belongs to a cycle of $\pi$ of length 2, so candidate $1_1$ does not belong to a committee fixed by $(\sigma,\pi)$.
On the other hand, the orbit of $1_2$ is $1_2 \leftrightarrow 2_1$. Thus, $1_2$ is on a fixed committee, and the orbit gives the other committee member. Therefore committee $\textbf{c}=(1_2,2_1)$ is fixed by the action of $((123),(12)(3);(12))$.

In the notation of Equation \eqref{eq:gnu}, $\pi=(12)$ has one cycle, and 
\begin{equation*}
g_1(\sigma;\pi) = \sigma_1\sigma_{\pi^{-1}(1)}=\sigma_1\sigma_2=(123)(12)(3)=(13)(2).\end{equation*}
We see that 2 is the unique fixed point of $g_1$, thus indicating that candidate 2 from department 1, along with all candidates in its orbit are part of a committee fixed by the action of $(\sigma,\pi)$. 

The action of $(e,e;(12)) \in S_3 \wr S_2$ sends candidates in department 1 to the corresponding candidate in department 2. That is, $1_k$ is mapped to $2_k$ for $k\in [3]$.
There is one cycle of $\pi=(12)$ and we can write  $$g_1(e,e;(12)) =\sigma_1 \sigma_2 = ee=e.$$
Since every candidate in department 1 is a fixed point of $g_1$, each choice of $1_1$, $1_2$, or $1_3$ determines a fixed committee.  Thus  $(1_1,2_1)$, $(1_2,2_2)$, $(1_3,2_3)$ are the three committees fixed by $(e,e;(12))$.

Finally, consider $((132),(12)(3),(123),(132);(12)(34))=(\sigma_1,\sigma_2,\sigma_3,\sigma_4;\pi) \in S_3 \wr S_4$. Here, $\pi$ has two cycles; we will order $(12)$ first and $(34)$ second. This gives
$$g_1(\sigma;\pi) =  \sigma_1 \sigma_{\pi^{-1}(1)} = \sigma_1\sigma_2 = (132)(12)(3) = (1)(23),$$
$$g_2(\sigma;\pi) = \sigma_3\sigma_{\pi^{-1}(3)} = \sigma_3\sigma_4 = (123)(132) = (1)(2)(3). $$
The fixed point of $g_1$ corresponds to candidate $\{1_1\}$ and the fixed points of $g_2$ correspond to candidates $\{3_1, 3_2, 3_3\}$ (notice, the candidates belong to department 3, as we began $g_2$ with $\sigma_3$). Here the orbits of the fixed candidates are $1_1 \leftrightarrow 2_2$, $3_1 \leftrightarrow 4_3$, $3_2 \leftrightarrow 4_1$, and $3_3 \leftrightarrow 4_2$. 
Choosing one candidate from each fixed set, and filling out the committee using the orbits of the chosen fixed candidates gives the three committees fixed by $(\sigma;\pi)$:
$$\{(1_1, 2_2, 3_1, 4_3), (1_1, 2_2, 3_2, 4_1), (1_1, 2_2, 3_3, 4_2)\}$$
\end{example}

Since the set of committees forms a basis for the $\Q S_m\wr S_n$-module $R$, we can evaluate the character of the corresponding permutation representation for the element $(\sigma;\pi)$ by counting the number of fixed committees under the action of $(\sigma;\pi).$  This leads to the following lemma.

\begin{lemma}\label{lemma: charR}
    Let $\chi_R$ denote the character of the permutation representation of $S_m\wr S_n$ corresponding the results space $R$, and $\chi_{\N}$ denote the character of the natural representation of $S_m$, then for all $(\sigma;\pi)\in S_m\wr S_n$,
\begin{equation}\label{charR}
     \chi_R(\sigma;\pi)=\prod_{\nu=1}^{c(\pi)}\chi_{\N}(g_{\nu}(\sigma;\pi)).   
\end{equation}
\end{lemma}

\begin{proof}
Since the representation of $S_m\wr S_n$ associated with the results space $R$ is a permutation representation built on a basis indexed by the set of all possible committees, the character $\chi_R(\sigma;\pi)$ is equal to the number of committees fixed by $(\sigma;\pi)$.  As explained in the previous discussion, the number of committees fixed under the action of $(\sigma;\pi)$ is equal to the product over $\nu$ of the number of fixed points of $g_{\nu}$.  Recalling  $g_{\nu}$ is an element of $S_m$, and that the character of the natural representation of $S_m$ counts the number of fixed points for each element, the result follows.
\end{proof}

\subsection{Decompositions of wreath product representations}\label{WreathRep}

Theorem \ref{thm: profile space decomposition1} and Lemma \ref{lemma: charR} reveal that the representations corresponding to the profile and results spaces of a neutral voting procedure are related to the regular representation of $S_m\wr S_n$ and the natural representation of $S_m$, respectively.  As we will show, this allows us to decompose our profile and results spaces into simple submodules defined in terms of simple $S_m$- and $S_n$-modules.

Up to isomorphism, there are finitely many distinct simple $S_m$-modules.  
It is well known that over fields of characteristic zero, the simple modules of the symmetric group $S_m$ are 
indexed by the partitions $\mu$ of $m$, and denoted $S^\mu$. 
 Thus there are $p(m)$ simple $S_m$-modules, where $p(m)$ is the number of partitions of $m$. 
In what follows, we will use the standard convention of denoting both the simple modules of the symmetric group and their corresponding representations of $S_m$ by $S^{\mu}.$ Whether $S^{\mu}$ refers to the representation or the underlying module should be inferred from the context.

The natural representation of $S_m$ decomposes into two irreducible representations, the trivial representation $S^{(m)}$ and the $(m-1)$-dimensional representation $S^{(m-1,1)}$.  The regular representation of any group decomposes into the direct sum of all of its irreducible representations repeated according to the dimension of the corresponding module \cite[Corollary 2.18]{FH}.  Thus, in the case of $S_m$, the module $\mathcal{R}_{S_m}$ corresponding to the regular representation  decomposes as  
$$\R_{S_m}\cong\bigoplus_{\mu\vdash m}(S^{\mu})^{\oplus \dim(S^\mu)}.$$

Given an $n$-multiset
$D_1,\ldots,D_n$ of representations of $S_m$, the tensor product $D^*=D_1 \otimes \dots \otimes D_n$ forms a representation of $S_m^n.$ If the representations of $S_m$ are all irreducible, the tensor product $D^*$ will similarly be an irreducible representation of $S_m^n.$ The multiset of representations defines a partition $\alpha=(\alpha_1, \ldots, \alpha_{\ell})$ of $n$ based on the number of repeated representations of $S_m$ of each type.  For example, if we take the irreducible representation $$D^*=S^{\begin{tiny}\yng(3)\end{tiny}}\otimes S^{\begin{tiny}\yng(1,1,1)\end{tiny}}\otimes S^{\begin{tiny}\yng(3)\end{tiny}}\otimes S^{\begin{tiny}\yng(3)\end{tiny}}$$ of $S_3^4$, the associated partition of $4$ is $\alpha=(3,1).$  The inertia group of $D^*$ is the subgroup $S_m\wr Y \unlhd S_m\wr S_n,$ where $Y$ is the Young subgroup $S_{\alpha_1}\times \dots \times S_{\alpha_{\ell}}$ associated with the partition $\alpha$.
The representation $D^*$ of $S_m^n$ extends to a representation $D^\sim$ of its inertia group via composition with a suitable permutation representation of $S_n$ \cite[4.3.30]{JK}.  

If $D^*$ is an irreducible representation of $S_m^n$, the extension $D^{\sim}$ will be an irreducible representation of the inertia group.  Similarly, an irreducible representation $D_Y$ of the Young subgroup $Y$ can be extended to an irreducible representation $D'$ of the inertia group $S_m\wr Y$ by defining $$D'(\sigma,\pi)=D_Y(\pi),\qquad \mbox{for all }(\sigma,\pi)\in S_m\wr Y.$$  The tensor product $D^{\sim}\otimes D'$ of irreducible representations is also irreducible, and lifts to an irreducible representation of $S_m\wr S_n.$  As the following theorem states, all irreducible representations of $S_m\wr S_n$ are of this form.

We note that James and Kerber assume representations are over an algebraically closed field in Theorem \ref{jkprc}. However, since the rationals are a splitting field for $S_m \wr S_n$ \cite[4.4.9]{JK}, any representation of $S_m \wr S_n$ over the complex numbers is equivalent to a representation over $\QQ$.

\begin{theorem}\cite[4.3.34]{JK}\label{jkprc}
Let $S^{\mu_1},\ldots,S^{\mu_n}$ be an $n$-multiset of irreducible representations of $S_m$ with associated Young subgroup $Y=S_{\alpha_1}\times S_{\alpha_2} \times \cdots \times S_{\alpha_{\ell}}$ of $S_n$. Then 
$$D^{\sim}\otimes D'\uparrow_{S_m\wr Y}^{S_m \wr S_n}$$ is an irreducible representation of $S_m\wr S_n,$ where $D^{\sim}$ and $D'$ are irreducible representations of $S_m\wr Y$ found by extending the irreducible representation $D^*=S^{\mu_1} \otimes \dots \otimes S^{\mu_n}$ of $S_m^n$ and an irreducible representation of $Y$ respectively. Furthermore, every irreducible representation of $S_m\wr S_n$ is of this form.
\end{theorem}

Thus, every irreducible representation of $S_m\wr S_n$ corresponds to an $n$-multiset of irreducible representations of $S_m$ taken together with an irreducible representation of a Young subgroup of $S_n$. With this in mind, we can denote the irreducible representations of $S_m\wr S_n$ (or their corresponding modules) as follows. 
Fix an ordering of the $p(m)$ irreducible representations of $S_m$.  
Let $\pmb{\lambda}=(\lambda_1,\lambda_2,\mathellipsis,\lambda_{p(m)})$ denote a vector of dimension $p(m)$ whose components $\{\lambda_i\}_{i=1}^{p(m)}$ are partitions (possibly empty) such that $\sum_{i=1}^{p(m)} |\lambda_i|=n$.
Each component holds two different kinds of information, as demonstrated in Example \ref{lambdaEX}.  First, $|\lambda_i|$ indicates the number of copies of the corresponding irreducible $S_m$ representation included in the representation $D^*$ of $S_m^n$.  Second, the shapes of the components indicate the irreducible representations of symmetric groups comprising the representation of the Young subgroup.  
 
\begin{example}\label{lambdaEX}
The irreducible representation $S^{\pmb{\lambda}}$ of $S_3 \wr S_4$ indexed by the vector $\pmb{\lambda}=\left(\begin{tiny}\yng(2,1),\yng(1)\end{tiny},\emptyset\right)$ is $$\left[\left((S^{\begin{tiny}\yng(3)\end{tiny}})^{\otimes 3}\otimes S^{\begin{tiny}\yng(1,1,1)\end{tiny}}\right)  \Big\uparrow_{S_3^4}^{S_3 \wr (S_3\times S_1)} \otimes \left(S^{\begin{tiny}\yng(2,1)\end{tiny}}\otimes S^{\begin{tiny}\yng(1)\end{tiny}}\right)\Big\uparrow_{(S_3\times S_1)}^{S_3 \wr (S_3\times S_1)}
\right] \Big\uparrow_{S_3 \wr (S_3\times S_1)}
^{S_3 \wr S_4}
$$ 
where we have taken $S^{\begin{tiny}\yng(3)\end{tiny}}$, $S^{\begin{tiny}\yng(1,1,1)\end{tiny}}$, $S^{\begin{tiny}\yng(2,1)\end{tiny}}$ to be our fixed ordering of the irreducible representations of $S_3$,  and $S^{\begin{tiny}\yng(2,1)\end{tiny}}\otimes S^{\begin{tiny}\yng(1)\end{tiny}}$ is an irreducible representation of the Young subgroup $S_3\times S_1$.
\end{example}

In the special case where $D^*=\bigotimes^n D$ for some representation $D$ of $S_m$, the inertia group of $D^*$ is all of $S_m\wr S_n$.  Thus, $D^*$ extends directly to a representation $D^{\sim}$ of $S_m\wr S_n$ denoted $$(\bigotimes^n D)\uparrow_{S_m^n}^{S_m \wr S_n}.$$  Furthermore, by the following lemma, the character of $(\bigotimes^n D)\uparrow_{S_m^n}^{S_m \wr S_n}$ can be calculated from the character of the representation $D$ of $S_m$ and the cycle product $g_{\nu}$ defined in Equation \ref{eq:gnu}.

\begin{lemma}\cite[4.3.9]{JK}\label{lemma: D-character}
    For all $(\sigma;\pi)\in\smsn$, given a representation $D$ of $S_m$, the character of $(\bigotimes^n D)\uparrow_{S_m^n}^{S_m \wr S_n}$ is:
    \begin{equation} \label{eqn: charLift}
        \chi_{(\bigotimes^n D)\uparrow_{S_m^n}^{S_m \wr S_n}}(\sigma;\pi) = \prod_{\nu = 1}^{c(\pi)}\chi_D(g_{\nu}(\sigma;\pi)).
    \end{equation}
\end{lemma}

Setting $D=\N$, we find Equation \ref{eqn: charLift} is the same as that found for the character associated with the results space in Lemma \ref{lemma: charR}.  Thus, by the correspondence between representations and modules, and the fact that representations are uniquely  determined by their characters over fields of characteristic zero, it follows that the results space $R$ is isomorphic to the $\Q S_m\wr~S_n$-module associated with the lifted tensor product of the natural representation, $
    (\bigotimes^n \N)\uparrow_{S_m^n}^{S_m \wr S_n}.$

\subsection{Profile and Results space decomposition}
Using the partition system of cataloging the irreducible $\QQ S_m$- and $\QQ S_m\wr S_n$- modules (or the corresponding representations) described in Section \ref{WreathRep}, we decompose the profile and results spaces into their irreducible components.

\begin{theorem}
\label{thm: profile space decomposition}
The profile space $P$ for a neutral full-ranking voting procedure has the following decomposition into simple submodules:
\begin{equation}
   P \cong \bigoplus_{\frac{m^n!}{(m!)^n n!}} \bigoplus_{\pmb{\lambda}} (S^{\pmb{\lambda}})^{\oplus \dim(S^{\pmb{\lambda}})} 
\end{equation}
where the inner sum $\oplus_{\pmb{\lambda}}$ runs over 
a single copy of each simple $\QQ S_m \wr S_n$-module $S^{\pmb{\lambda}}$. 
\end{theorem}

\begin{proof}  Recall from Theorem \ref{thm: profile space decomposition1}, the profile space is given by 
$$
   P \cong \bigoplus_{\frac{m^n!}{(m!)^n n!}} \R_{S_m\wr S_n},
$$
where $\R_{S_m\wr S_n}$ is the regular $\QQ S_m \wr S_n$-module.  The result then follows directly by Theorem \ref{jkprc} and the fact that regular representations decompose into the direct sum of irreducible representations repeated with multiplicity equal to the dimension of the representation.
\end{proof}

\begin{theorem}
\label{thm: results space decomposition}
The results space $R$ has the following decomposition into simple $\QQ S_m \wr S_n$-submodules:
 \begin{equation}
     R \cong \bigoplus_{k=0}^n\bigoplus_{\binom{n}{k}} S^{((n-k),(k),\emptyset,\ldots,\emptyset)}.
 \end{equation}
\end{theorem}

\begin{proof}

Let $D^R$ denote the representation corresponding to the results space of a neutral voting procedure on committees.   As noted at the end of Section \ref{WreathRep}, 
$D^R\cong(\bigotimes^n \N)\uparrow_{S_m^n}^{S_m \wr S_n},$
where $\N$ is the natural representation of $S_m$. 
Decomposing $\N$ gives 
$D^R \cong \left(\bigotimes^n (S^{(m)} \oplus S^{(m-1,1)})\right)\uparrow_{S_m^n}^{S_m \wr S_n}$, and distributing the tensor product of the direct sum we find 
$$D^R \cong \left(\bigoplus_{k=0}^n\bigoplus_{\binom{n}{k}} \left(S^{(m)^{\otimes^{n-k}} }\otimes S^{(m-1,1)^{\otimes^{k}}}\right)\right)\Big\uparrow_{S_m^n}^{S_m \wr S_n}.$$

We can find the lift of this direct sum by taking the direct sum of each term lifted.  
For each $k$, denote the irreducible representation $S^{(m)^{\otimes^{n-k}}} \otimes S^{(m-1,1)^{\otimes^{k}}}$ of $S_m^n$ by $D_k^*.$  Thus, $D^R\cong \bigoplus_{k=0}^n\bigoplus_{\binom{n}{k}} D_k^* \uparrow_{S_m^n}^{S_m\wr S_n}.$  As described in Section \ref{WreathRep}, we can first lift the representation $D_k^*$ to $S_m\wr Y_k,$ where $Y_k$ is the Young subgroup associated with $D_k^*,$ and then to $S_m\wr S_n$. Thus, the representation corresponding to the results space can be expressed

\begin{equation}\label{eq: R rep decomp1}
D^R \cong \bigoplus_{k=0}^n\bigoplus_{\binom{n}{k}}\left(
D_k^*
\big\uparrow_{S_m^n}^{S_m\wr Y_k}
\right)\Big\uparrow_{S_m\wr Y_k}^{S_m \wr S_n}.
\end{equation}

For a given $k$, the Young subgroup associated with $D_k^*$ is $S_{n-k}\times S_k$.  Furthermore, tensoring by the trivial representations $\left(S^{(n-k)}\otimes S^{(k)}\right)\uparrow_{Y_k}^{S_m\wr Y_k}$ does not change the representation. Therefore
$\left(D^*_k \big\uparrow_{S_m^n}^{S_m\wr Y_k}\right)\big\uparrow_{S_m\wr Y_k}^{S_m\wr S_n}$ is equal to 
\begin{equation}\label{eq: D_k lift}
\left(\left( S^{(m)^{\otimes^{n-k}} }\otimes S^{(m-1,1)^{\otimes^{k}}}\right)\big\uparrow_{S_m^n}^{S_m\wr Y_k}\otimes \left(S^{(n-k)}\otimes S^{(k)}\right)\big\uparrow_{Y_k}^{S_m\wr Y_k}\right)\big\uparrow_{S_m\wr Y_k}^{S_m\wr S_n},
\end{equation}
which we recognize from Section \ref{WreathRep} as the irreducible representation of $S_m\wr~S_n$ indexed by the vector $\pmb{\lambda}=((n-k),(k),\emptyset,\ldots,\emptyset)$ under the fixed ordering $S^{(m)}$, $S^{(m-1,1)}$, $\ldots,S^{(1,1,\ldots,1)}$ of the irreducible representations of of $S_m.$  Thus, by Equations \ref{eq: R rep decomp1} and \ref{eq: D_k lift}, we have an expression for the representation $D^R$ as the direct sum of the irreducible representations $S^{((n-k),(k),\emptyset,\ldots,\emptyset)}$ of $S_m\wr S_n,$ repeated with multiplicity equal to the binomial coefficient $\binom{n}{k}$.  It follows that the results space decomposes into the direct sum of the corresponding simple $\QQ S_m\wr S_n$-modules, i.e.,
$$R \cong \bigoplus_{k=0}^n\bigoplus_{\binom{n}{k}} S^{((n-k),(k),\emptyset,\ldots,\emptyset)}.$$
\end{proof}

Theorem \ref{thm: results space decomposition} was partially anticipated by Lee who noticed that in the $S_2\wr S_n$ case the partitions appearing in $\pmb{\lambda}$ are ``flat''. For example, when $n=5$ and $k=2$, 
$\pmb{\lambda}=\left(\begin{tiny}\yng(3),\yng(2)\end{tiny}\right).$
Lee conjectured that the decomposition of the results space in the $S_2 \wr S_n$ case consisted of the sum over all submodules $S^{(\mu,\nu)}$ where both $\mu$ and $\nu$ are trivial (``flat'') partitions \cite{Lee}. Lee's conjecture was subsequently proven by Davis \cite{Davis}. 

\vspace{.25in}

 Recall from Section~\ref{sec: alg voting background} that we view a positional voting procedure $T_{\mathbf{w}}$ with (nonzero) weighting vector $\mathbf{w} \in \QQ^{m^n}$ as a $\QQ S_m \wr S_n$-module homomorphism from the profile space to the results space. Thus by Theorems~\ref{thm: profile space decomposition} and \ref{thm: results space decomposition}, we have the following (nonzero) module homomorphism between direct sums of simple modules:
\begin{equation}
    T_{\mathbf{w}} : \bigoplus_{\frac{m^n!}{(m!)^n n!}} \bigoplus_{\pmb{\lambda}} (S^{\pmb{\lambda}})^{\oplus \dim(S^{\pmb{\lambda}})} \to \bigoplus_{k=0}^n\bigoplus_{\binom{n}{k}} S^{((n-k),(k),\emptyset,\ldots,\emptyset)}.
\end{equation}

By Schur's Lemma, 
the kernel of $T_{\mathbf{w}}$ contains all simple $\QQ S_m \wr S_n$-modules $S^{\pmb{\lambda}}$ where 
$\pmb{\lambda} \neq ((n-k),(k),\emptyset, \ldots, \emptyset)$ for $k \in \{0,1,\ldots,n\}$.
The voter preference information contained in these submodules, therefore, has no impact on the election results. 

 \begin{example} Consider an election for a committee formed by choosing one of three candidates from two different departments.
We order the $9$ committees lexicographically, so that a vector $[8,5,2,1,7,4,3,0,6]^t$ in the results space represents 8 points for committee $(1_1,2_1)$, 5 points for $(1_1,2_2)$, etc.
 
 As an $S_3 \wr S_2$-module, the results space $R$ decomposes into the direct sum of irreducible submodules 
  $R \cong S^{(\begin{tiny}\yng(2),\emptyset,\emptyset\end{tiny})}  \oplus S^{(\begin{tiny}\yng(1),\yng(1),\emptyset\end{tiny})}\oplus S^{(\begin{tiny}\yng(1),\yng(1),\emptyset\end{tiny})}\oplus S^{(\begin{tiny}\emptyset,\yng(2),\emptyset\end{tiny})}$ where 

 $S^{(\begin{tiny}\yng(2),\emptyset,\emptyset\end{tiny})} =\left<
 \begin{bmatrix}
1 \\
1 \\
1 \\
1 \\
1 \\
1\\
1\\
1\\
1\\
\end{bmatrix}
\right>$,
\hfill
 $S^{(\begin{tiny}\emptyset,\yng(2),\emptyset\end{tiny})} =\left<
 \begin{bmatrix}
4 \\
-2 \\
-2 \\
-2 \\
1 \\
1\\
-2\\
1\\
1\\
\end{bmatrix}, 
\begin{bmatrix}
-2 \\
 4\\
 -2\\
 1\\
 -2\\
1\\
1\\
-2\\
1\\
\end{bmatrix}, 
\begin{bmatrix}
-2 \\
 1\\
 1\\
 4\\
 -2\\
-2\\
-2\\
1\\
1\\
\end{bmatrix}, 
\begin{bmatrix}
 1\\
  -2\\
  1\\
  -2\\
  4\\
 -2\\
 1\\
-2 \\
 1\\
\end{bmatrix}
\right>,$ and

$$S^{(\begin{tiny}\yng(1),\yng(1),\emptyset\end{tiny})} \oplus S^{(\begin{tiny}\yng(1),\yng(1),\emptyset\end{tiny})}  =\left<
 \begin{bmatrix}
4 \\
1 \\
1 \\
1 \\
-2 \\
-2\\
1\\
-2\\
-2\\
\end{bmatrix}, 
\begin{bmatrix}
1 \\
 4\\
 1\\
 -2\\
1\\
-2\\
-2\\
1\\
-2\\
\end{bmatrix}, 
\begin{bmatrix}
1 \\
 1\\
 4\\
-2\\
 -2\\
1\\
-2\\
-2\\
1\\
\end{bmatrix}, 
\begin{bmatrix}
 1\\
  -2\\
  -2\\
4\\
 1\\
 1\\
 1\\
-2 \\
-2\\
\end{bmatrix}
\right>.$$

These vectors were found using character tables of wreath product elements \cite{Sanger} and the following decomposition algorithm [Prop. 14.26] of \cite{JL}:
For $\chi$ an irreducible character of a group $G$, and $V$ a $\mathbb{C} G$-module, the sum of the $\mathbb{C}G$-submodules of $V$ with character $\chi$ is given by
$$\left( \sum_{g\in G}\chi(g^{-1})g\right)V.$$ 
Again, since the rationals are a splitting field for $S_m \wr S_n$, this algorithm applies to the decomposition of the results space, viewed as a $\QQ S_m\wr S_n$-module.

Notice that results vectors in the two isomorphic $S^{(\begin{tiny}\yng(1),\yng(1),\emptyset\end{tiny})}$ submodules indicate an electorate that shows strong support for a favorite committee (corresponding to an entry of 4), tepid support (values of 1) for a committee that shares one candidate 
with the preferred committee, and a lack of support (values of -2) for committees that are disjoint from their preferred committee. 

Conversely, results in $S^{(\emptyset, \begin{tiny}\yng(2),\emptyset\end{tiny})}$ indicate voters who support a preferred committee (value 4) but do not support committees that share one candidate with the preferred committee (value -2). Committees whose members are disjoint from the preferred committee receive tepid support (value 1). 
 \end{example}

\section{A voting paradox}
\label{sec: voting paradox}

In this section we prove a committee selection analog to \cite{DEMO}'s primary voting paradox result. 
As described in the introduction, these results essentially show that so long as a choice of weighting vectors are ``different enough", the associated positional voting procedures can yield radically different outcomes. We make this notion of ``different enough" more precise after stating the theorem from \cite{DEMO}:

\begin{theorem}
{\cite[Theorem~1]{DEMO}}
\label{thm: DEMO 1}
Let $n \geq 2$. 
Suppose that $\mathbf{w}_1, \ldots, \mathbf{w}_j$ form a linearly independent set of sum-zero weighting vectors. 
If $\mathbf{r}_1, \ldots, \mathbf{r}_j$ are any 
results vectors whose entries sum to zero, 
then there exist infinitely many profiles $\mathbf{p}$ 
such that $T_{\mathbf{w}_i}(\mathbf{p})=\mathbf{r}_i$ for all $1 \leq i \leq j$. 
\end{theorem}

In \cite{DEMO}, the profile space $P$ and results space $R$ were naturally viewed as $\QQ S_n$-modules. As such, $R \cong S^{(n)} \oplus S^{(n-1,1)}$ decomposes into irreducibles as the direct sum of the trivial module $S^{(n)}$ and a single other irreducible, $S^{(n-1,1)}$, spanned by sum-zero vectors.
 Recalling that weighting vectors can be viewed as elements of the results space and each weighting vector $\mathbf{w}$ can be decomposed into  $\mathbf{w}=\alpha\mathbf{1} +\mathbf{\hat{w}}$, where all of the information differentiating the candidates is contained in the sum-zero portion  $\mathbf{\hat{w}}$,  
 weight vectors are ``different enough" in the candidate-selection setting so long as they are linearly independent in $S^{(n-1,1)}$.

For the committee-selection positional voting procedures we analyze here, the profile and results spaces are naturally viewed as $\QQ S_m \wr S_n$-modules.
These spaces exhibit more complex behavior, and accordingly we must place considerably stronger conditions on the weighting vectors to ensure they are ``different enough." 
We proved in Theorem~\ref{thm: results space decomposition} that the results space has irreducible decomposition 
$R \cong \oplus_{k=0}^n \oplus_{\binom{n}{k}}S^{((n-k), (k), \emptyset, \ldots, \emptyset)}$, where the information that determines the outcome of the election is contained in the projection of the weighting vector onto the submodules $\oplus_{k=1}^n \oplus_{\binom{n}{k}} S^{((n-k),(k),\emptyset, \ldots,\emptyset)}$. 
Certainly the basis vectors for this direct sum of irreducibles are all sum-zero; however, since there are several irreducible modules rather than a single one we must require that the weighting vectors be linearly independent {\em in each}. This gives us the following:

\begin{theorem}
\label{thm: analog to DEMO 1}
Let $n \geq 2$. 
Suppose that $\mathbf{w}_1, \ldots, \mathbf{w}_j$ form a set of sum-zero weighting vectors such that for each $k \in [n]$ their projections $\Proj_k(\mathbf{w}_1, \ldots, \mathbf{w}_j)$ onto $S^{((n-k),(k),\emptyset,\ldots,\emptyset)}$ are linearly independent.
If  $\mathbf{r}_1, \ldots, \mathbf{r}_j$ are any results vectors whose entries sum to zero, then there exist infinitely many profiles $\mathbf{p}$ such that $T_{\mathbf{w}_i}(\mathbf{p})=\mathbf{r}_i$ for all $1 \leq i \leq j$. 
\end{theorem}

In order to prove Theorem \ref{thm: analog to DEMO 1}, it will be useful to 
view a positional voting method as the action of a profile (viewed as a symmetric group algebra element) on a weighting vector in a manner which we now describe. 
Observe that we can view profiles as elements of the group algebra $\QQ S_{m^n}$ by way of a natural bijection between the full rankings which form the basis of the profile space and permutations of the $m^n$ committees. 
In particular, let the identity $e \in S_{m^n}$ correspond to the full ranking wherein committees, viewed as vectors $(1_{k_1}, 2_{k_2}, \ldots, n_{k_n})$, are ordered lexicographically. Define the correspondence on the other group elements accordingly, by permutations of the identity ranking.

Recall that weighting vectors can be viewed as elements of the results space $R$ with basis given by committees. Ordering this basis lexicographically as well, we can view our positional voting procedures as the results of profiles (viewed as elements of $\QQ S_{m^n}$) acting on weighting vectors. That is,
$$T_{\ww}(\mathbf{p})=\mathbf{p}\ww.$$
We note that since $\QQ S_m\wr S_n$ is isomorphic to a subalgebra of $\QQ S_{m^n}$, in some cases $\mathbf{p}\ww$ can be realized as the action $a \ww$ for some $a\in \QQ S_m \wr S_n.$

\begin{example}
Returning to the setup of Example~2.1 with Borda count weighting vector $\ww=[3,2,1,0]^t$ and profile vector $\mathbf{p}=[4,0,3,0,0,0,0,2,0,\ldots,0]^t$, we have

\[T_{[3,2,1,0]}(\mathbf{p}) =
\begin{bmatrix}
  3 & 3 & 3 & 3 & 3 & 3 & 2 & 2 \cdots\\
  2 & 2 & 1 & 0 & 1 & 0 & 3 & 3 \cdots\\
  1 & 0 & 2 & 2 & 0 & 1 & 1 & 0 \cdots \\
  0 & 1 & 0 & 1 & 2 & 2 & 0 & 1 \cdots\\
  \end{bmatrix}
\begin{bmatrix}
4 \\
0 \\
3 \\
0 \\
0 \\
0\\
0\\
2\\
\vdots
\end{bmatrix}.
 \]
Recalling that the columns in the matrix $T_{\ww}$ are permutations of $\ww \in \QQ^{4}$, we can reinterpret this application of the linear transformation $T_{\mathbf{w}}$ to $\mathbf{p} \in \Q^{4!}$ by viewing  $\mathbf{p}$ as an element of $\Q S_{4}$ acting on ${\ww}$ as follows: 
\begin{align*}
T_{\mathbf{w}}(\mathbf{p})&=
4\begin{bmatrix}3 \\2\\1\\0 \end{bmatrix} + 3\begin{bmatrix}3\\1\\2\\0\end{bmatrix}
+2 \begin{bmatrix}2\\3\\0\\1\end{bmatrix}
\\
 &= \left(4\cdot e + 3 \cdot (23) + 2 \cdot (12)(34)\right)
 \begin{bmatrix}3 \\2\\1\\0 \end{bmatrix}
\\
 &=\mathbf{p}\mathbf{w}.
 \end{align*}
 
 Furthermore, the element $\left(4\cdot e + 3 \cdot (23) + 2 \cdot (12)(34)\right) \in \Q S_{4}$ has a corresponding element in $\Q S_2 \wr S_2$:
\begin{align*}
T_{\mathbf{w}}(\mathbf{p})
&=\left(4\cdot (e,e;e) + 3 \cdot (e,e;(12)) + 2 \cdot (e,(12);e)\right)
 \begin{bmatrix}3 \\2\\1\\0 \end{bmatrix}
 \end{align*}
  
Thus the profile $\mathbf{p}=[4,0,3,0,0,0,0,2,0,\ldots,0]^t$ corresponds to 
$$4e+3(23)+2(12)(34) \in \QQ S_4$$
which corresponds to 
$$a= 4(e,e;e) + 3(e,e;(12)) + 2(e,(12);e) \in \QQ S_2 \wr S_2.$$
\end{example} 
We are now ready to prove the main result of the section.
\begin{proof}[Proof of Theorem~\ref{thm: analog to DEMO 1}]
For each $k \in \{0,1,\ldots, n\}$, set $S^k := S^{((n-k),(k),\emptyset, \ldots, \emptyset)}$.  Recall that $S^0$ is the trivial module, and does not influence the outcome of the election. 
Thus we focus on $k \neq 0$ and restrict our attention to sum-zero weight vectors. 
Fix $k \in [n]$,
and for each $i \in [j]$, define $\ww_i^k$ and 
$\rr_i^k$ to be the projections of weighting vector 
$\ww_i$ and results vector $\rr_i$, respectively, onto the module $S^k$. That is, 
$\ww_i^k := \Proj_{S^k}\ww_i$, and 
$\rr_i^k := \Proj_{S^k}\rr_i$.

Next, define the matrices $W^k :=[\ww_1^k | \cdots | \ww_j^k]$ and $R^k = [\rr_1^k | \cdots | \rr^k_j]$ whose column vectors are the projected weighting and results vectors, respectively. 
By our assumption that the set $\{\ww_1^k, \ldots, \ww_j^k\}$ is linearly independent, it follows that the matrix $W^k$ has full column rank, and therefore has a left inverse $W^k_L$ such that $W^k_L W^k = I_j$, the identity matrix of size $j \times j$. Defining the linear transformation $L_k \in \End_{\QQ}(S^k)$ in matrix form by $L_k := R^k W^k_L$, we have $L_kW^k=R^k$. That is, $L_k(\ww_i^k)=\rr^k_i$ for each $i \in [j]$.

Finally, we define a linear transformation $L \in \oplus_{\pmb{\lambda}}\End_{\QQ}(S^{\pmb{\lambda}})$ by taking a single copy of each $L_k=L_{((n-k),(k),\emptyset,\ldots,\emptyset)}$ for $k \in [n]$, and the zero map on the one-dimensional module $S^0$ as well as on every simple submodule not in the results space $R$. 
That is, if we index so that 
$$\bigoplus_{\pmb{\lambda}} \End_{\QQ}(S^{\pmb{\lambda}}) = \End_{\QQ}(S^0) \oplus \End_{\QQ}(S^1) \oplus \ldots \oplus \End_{\QQ}(S^n) \oplus \bigoplus_{S^{\pmb{\lambda}} \notin R}\End_{\QQ}(S^{\pmb{\lambda}})$$
we define 
$L= 0 \oplus L_1 \oplus \cdots \oplus L_n \oplus 0 \oplus \ldots \oplus 0$.

Then for any vector $\ww \in \oplus_{\pmb{\lambda}} S^{\pmb{\lambda}}$ with decomposition $\ww = \oplus_{\pmb{\lambda}} \ww^{\pmb{\lambda}}$, we have
\begin{equation*}
L(\ww) = \bigoplus_{\pmb{\lambda}} L_{\pmb{\lambda}}(\ww^{\pmb{\lambda}}).
\end{equation*}
In particular, for each weighting vector $\ww_i$,
\begin{equation*}
L(\ww_i) = \bigoplus_{\pmb{\lambda}} L_{\pmb{\lambda}}(\ww_i^{\pmb{\lambda}}) = \oplus_{k=1}^n L_k(\ww_i^k) \oplus \bigoplus_{S^{\pmb{\lambda}} \notin R} 0 = \oplus_{k=1}^n \rr_i^k = \rr_i.
\end{equation*}

An application of the Density Theorem (see, e.g., \cite[Theorem~3.2.2 (ii)]{EGHLSVY}) guarantees there exists a group algebra element $a \in \QQ S_m \wr S_n$ 
such that 
\begin{equation}L(\ww) = a \ww \textup{ for all }\ww \in \oplus_{\pmb{\lambda}} S^{\pmb{\lambda}}.\end{equation} 
In particular, 
$$ a  \ww_i = \rr_i \textup{ for each }i \in [j].$$

From the discussion following Theorem \ref{thm: analog to DEMO 1},
there exists a profile $\mathbf{p} \in P$ corresponding to group algebra element $a \in \QQ S_m \wr S_n$ such that for any weighting vector $\ww$ viewed as an element of results space $R$, $a \ww = \mathbf{p}{\ww}$.  
So for each $i \in [j]$, we have
\begin{equation}
T_{\ww_i}(\pp) = \pp \ww_i = a \ww_i = \rr_i.
\end{equation}

That is, we have found a single profile $\pp$ that can yield wildly different election outcomes (corresponding to our arbitrary choices of results vectors $\rr_i$) under our different election procedures $T_{\ww_i}$.

Consider the linear transformation $M \in \oplus_{\pmb{\lambda}} \End_{\QQ}(S^{\pmb{\lambda}})$ 
defined by taking the zero map on each simple submodule $S^j$ in $R$ and the identity map on all other simple submodules $S^{\pmb{\lambda}}$. That is, $M = (1,0, \ldots, 0, I, \ldots, I)$.
Under this construction, $M(\ww_i)=0$ for each $i \in [j]$ since each $\ww_i \in R$.

However since $M \neq 0$, by the Density Theorem there must exist some $0 \neq b \in \QQ S_m \wr S_n$ such that for any $\ww \in \oplus_{\pmb{\lambda}} S^{\pmb{\lambda}}$, we have $ b \ww=M(\ww)$, so that in particular, $ b \ww_i = 0$ for each $i \in [j]$. 

Of course, once we have a single $0 \neq b \in \QQ S_m \wr S_n$ such that $b \ww_i =0$, we have infinitely many, as we may simply multiply $b$ by any scalar. This completes the proof: for any $c \in \QQ$, choose profile $\pp$ corresponding to $a + cb \in \QQ S_m \wr S_n$. Then
\begin{equation}
T_{\ww_i}(\pp)=\pp\ww_i = a \ww_i + cb \ww_i = \rr_i +0 = \rr_i, \textup{ for each }i \in [j].
\end{equation}
\end{proof}

\section{Acknowledgements}
This work was begun by the authors during the Banff International Research Institute (BIRS) workshop 17w5012 and continued at the Mathematical Sciences Research Institute (MSRI) in Berkeley, California, during the spring and summer of 2017. The BIRS workshop was partially supported by a National Science Foundation ADVANCE grant (award \# 1500481) to the Association for Women in Mathematics (AWM). The MSRI work was supported by the National Science Foundation under Grant No. DMS-1440140.   The second author was partially supported by the National Science Foundation under Grant No. DMS-1344199. The sixth author was partially supported by the National Science Foundation under Grant No. DMS-1500949.
The authors would also like to thank Matt Davis for helpful conversations.


\bigskip \bigskip

\bibliographystyle{amsalpha}

\end{document}